\documentclass[11pt]{amsart}
\usepackage{amssymb,latexsym,amsmath,graphicx,graphics,epic,eepic}

\theoremstyle{plain}
\newtheorem{theorem}{Theorem}
\numberwithin{theorem}{section}
\newtheorem{lemma}[theorem]{Lemma}

\newtheorem{corollary}[theorem]{Corollary}

\theoremstyle{definition}

\newcommand{\C}{{\mathbb C}}
\newcommand{\R}{{\mathbb R}}
\newcommand{\Z}{{\mathbb Z}}
\newcommand{\Q}{{\mathbb Q}}
\renewcommand{\P}{{\mathbb P}}
\newcommand{\s}{{\mathbb S}}

              \renewcommand{\L}{{\mathcal L}}

\begin{document}
\title{Hurwitz-type bound, knot surgery, and smooth $\s^1$-four-manifolds}
\author{Weimin Chen}
\subjclass[2000]{}
\thanks{The author is partially supported by NSF grant DMS-1065784.}
\date{\today}
\maketitle

\begin{abstract}
In this paper we prove several related results concerning smooth $\Z_p$ or $\s^1$ actions on 
$4$-manifolds.  We show that there exists an infinite sequence of smooth $4$-manifolds $X_n$,
$n\geq 2$, which have the same integral homology and intersection form and the same Seiberg-Witten invariant, such that each $X_n$ supports no smooth $\s^1$-actions but admits a smooth $\Z_n$-action.
In order to construct such manifolds, we devise a method for annihilating smooth
$\s^1$-actions on $4$-manifolds using Fintushel-Stern knot surgery, and apply it to the Kodaira-Thurston manifold in an equivariant setting. Finally, the method for annihilating smooth $\s^1$-actions relies 
on a new obstruction we derived in this paper for existence of smooth 
$\s^1$-actions on a $4$-manifold: the fundamental group of a smooth $\s^1$-four-manifold with nonzero Seiberg-Witten invariant must have infinite center. We also include a discussion on various analogous or related results in the literature, including locally linear actions or smooth actions in dimensions other than four.
\end{abstract}

\section{Introduction}

In this paper we discuss several related problems concerning existence or non-existence 
of smooth $\Z_p$ or $\s^1$ actions on a $4$-manifold 
(all group actions are assumed to be effective).

\vspace{3mm}

{\it Problem 1: Construct smooth $4$-manifolds with bounded topology, which support no smooth
$\s^1$-actions but admit smooth $\Z_p$-actions for arbitrarily large prime $p$.
}
\vspace{3mm}

This problem was motivated by the Hurwitz-type theorems concerning bound of automorphisms
of smooth projective varieties of general type (cf. \cite{Xiao, HMX}). In \cite{C0} the author initiated 
an investigation on an analogous question concerning bound of periodic diffeomorphisms 
of $4$-manifolds. Namely:

\vspace{2mm}

{\it Suppose a smooth $4$-manifold $X$ supports no smooth $\s^1$-actions, is there a $C>0$ 
such that there are no smooth $\Z_p$-actions on $X$ for any prime $p>C$?
}
\vspace{2mm}

We shall call such a constant $C$ a {\it Hurwitz-type bound} of $X$ for smooth $\Z_p$-actions. An 
interesting discovery in \cite{C0} was that Hurwitz-type bound for smooth actions of a $4$-manifold 
no longer depends on the topology (e.g. homology) alone, but also on the smooth structure of the manifold. More concretely, a Hurwitz-type bound was established in \cite{C0} for holomorphic 
$\Z_p$-actions which depends only on the homology:
$$
p\leq C=c(1+b_1+b_2+|Tor H_2|)
$$
where $b_i$ is the $i$-th Betti number, $Tor H_2$ is the torsion subgroup of the second homology, and $c>0$ is a universal constant. However, for smooth, even symplectic, $\Z_p$-actions,
the order $p$ is no longer bounded by the homology alone. In fact, we showed in \cite{C0} that 
for every prime number $p\geq 5$, there is a smooth $4$-manifold $X_p$ homeomorphic to the rational elliptic surface $\C\P^2\# 9 \overline{\C\P^2}$, such that $X_p$ supports no smooth 
$\s^1$-actions but admits a smooth $\Z_p$-action, which even preserves a symplectic structure
on $X_p$. Moreover, regarding symplectic $\Z_p$-actions, a Hurwitz-type bound was established
in \cite{C0} which involves the Betti numbers and the canonical class of the manifold.
Note that the latter, being a Seiberg-Witten basic class, is an invariant of the smooth structure.
It is a natural question as whether such a Hurwitz-type bound exists for smooth $\Z_p$-actions
(cf. \cite{C,C0}). The following theorem provides a negative answer to this question.

\begin{theorem}
There exist smooth $4$-manifolds $X_n$, $n\geq 2$, which have the same integral homology, 
intersection form, and Seiberg-Witten invariant, such that each $X_n$ supports no smooth 
$\s^1$-actions but admits a smooth $\Z_n$-action. Moreover, the Seiberg-Witten invariant of $X_n$
is nonzero. 
\end{theorem}

\noindent{\bf Remarks}
(1) An analogous question in the locally linear category was considered by Kwasik in \cite{K}
where it was shown that the fake $\C\P^2$ admits a locally linear $\Z_p$-action for every odd $p$
but supports no locally linear $\Z_2$-actions; in particular, it supports no locally linear $\s^1$-actions. 
Edmonds \cite{Ed} extended the construction to all simply-connected $4$-manifolds, among
which those with even intersection form and nonzero signature do not admit any locally linear 
$\s^1$-actions by work of Fintushel \cite{F1}. So no Hurwitz-type bound exists for
locally linear actions in dimension four. 

(2) In higher dimensions, there are examples of manifolds which support no smooth $\s^1$-actions 
but admits smooth $\Z_p$-actions for infinitely many primes $p$. Indeed, Assadi and Burghelea showed 
in \cite{AB} that for any exotic $n$-sphere $\Sigma^n$, the connected sum with the $n$-torus, 
$\Sigma^n\# T^n$, does not support any smooth $\s^1$-actions. Since the group of homotopy 
$n$-spheres is of finite order, it follows easily that, by taking connected sum equivariantly with 
respect to some natural $\Z_p$-action on $T^n$, one obtains such examples. 

(3) In dimension three, a similar question was asked by Giffen and Thurston \cite{Kirby}
as whether, for a closed $3$-manifold $M$, the order of finite subgroups of $\text{Diff } M$
is bounded so that it contains no infinite torsion subgroups, unless $M$ admits an $\s^1$-action.
Kojima \cite{Kojima} answered this question affirmatively, modulo the Geometrization Conjecture
of Thurston which is now resolved \cite{BLP, P}.  A key point was the work of Freedman and Yau 
on homotopically trivial symmetries of Haken manifolds \cite{FY}.

\vspace{3mm}

The $4$-manifolds $X_n$ in Theorem 1.1 were obtained by performing Fintushel-Stern knot
surgery \cite{FinS} on the Kodaira-Thurston manifold $X$, which is a smooth $T^2$-manifold. 
The idea is to kill all the smooth $\s^1$-actions on $X$ while allowing a finite symmetry to survive by performing the knot surgery equivariantly with respect to some natural $\Z_n$-action on $X$. 
Suppressing the equivariant aspect of this construction, one may consider more generally the following problem.

\vspace{3mm}

{\it Problem 2: For any given smooth $\s^1$-four-manifold $X$, perform surgeries on $X$ so that
the surgery manifold does not support any smooth $\s^1$-actions while retain the topological
or smooth invariants of $X$ as much as possible. 
}

\vspace{3mm}

In this regard, we have the following theorem.

\begin{theorem}
Let $X$ be a smooth $\s^1$-four-manifold with $b_2^{+}>1$ and nonzero Seiberg-Witten
invariant, such that $H^2(X;\Z)$ has no $2$-torsions. Then there exists a Fintushel-Stern 
knot surgery manifold of $X$ which retains the integral homology, intersection form, and 
Seiberg-Witten invariant, but admits no smooth $\s^1$-actions. 
\end{theorem}

Here we call a $4$-manifold a {\it Fintushel-Stern knot surgery manifold} 
of $X$ if it is obtained by a finite series of knot surgeries starting from $X$. 

\vspace{3mm}

\noindent{\bf Remarks}
(1) By work of Pao \cite{Pao}, Fintushel \cite{F2}, and independently Yoshida \cite{Y} (modulo
the $3$-dimensional Poincar\'{e} conjecture which is now resolved by work of Perelman \cite{P}),
a simply-connected, smooth $\s^1$-four-manifold has standard smooth structure, i.e., it is 
diffeomorphic to a connected sum of $\s^4$, $\pm \C\P^2$, or $\s^2\times \s^2$. With this
understood, any surgery which changes the smooth structure of such a manifold will kill all 
the smooth $\s^1$-actions but retain the homeomorphism type of the manifold. However, such 
a procedure will change the Seiberg-Witten invariant in general.

(2) Kotschick in \cite{Ko2} constructed examples of pairs of smooth, homeomorphic but 
non-diffeomorphic $4$-manifolds, where one admits a smooth $\s^1$-action and the other does 
not support any smooth $\s^1$-actions. Kotschick's examples are homeomorphic to 
$k(\s^2\times \s^2)\# (k+1)(\s^1\times \s^3)$ and all have vanishing Seiberg-Witten invariant. 
The non-existence of smooth $\s^1$-actions was based on the existence of monopole classes 
$c$ with $c^2>0$, which is a smooth invariant finer than Seiberg-Witten basic class. (Kotschick only 
showed non-existence of fixed-point free actions in \cite{Ko2}; to rule out actions with fixed-points, 
one appeals to Theorem 2.1 of Baldridge \cite{Bald3}, cf. Kotschick \cite{Ko3}.)

(3) Freedman and Meeks \cite{FMk} derived a certain obstruction for existence of smooth
$\s^1$-actions in terms of certain properties of the de Rham cohomology. For instance, their 
obstruction implies that for any smooth $n$-manifold $M^n$ which is not a homotopy $n$-sphere, 
the connected sum $M^n\# T^n$ does not admit any smooth $\s^1$-actions.

\vspace{3mm}

The construction in Theorem 1.2 relies on a new obstruction for existence of smooth
$\s^1$-actions on a $4$-manifold, which is stated below.

\begin{theorem}
Let $X$ be a smooth $\s^1$-four-manifold with $b_2^{+}>1$ and nonzero Seiberg-Witten
invariant. Then the homotopy class of the principal orbits of any smooth $\s^1$-action on $X$
must have infinite order. In particular, the center of $\pi_1(X)$ is infinite.
\end{theorem}

\noindent{\bf Remarks} (1) In fact, the statement in Theorem 1.3 concerning the homotopy class
of principal orbits holds true more generally for locally linear actions. Namely: for a smooth 
$\s^1$-four-manifold $X$ with $b_2^{+}>1$ and nonzero Seiberg-Witten invariant, the homotopy 
class of the principal orbits of any locally linear $\s^1$-action on $X$ must have infinite order.
This generalization is the consequence of two separate results. The first one, which asserts that every 
locally linear $\s^1$-action on $X$ must be fixed-point free, is proved in this paper (cf. Theorem 2.6).
The second one is Theorem 1.6 in \cite{C2} which states that if the $\pi_1$ of
an orientable $4$-manifold has infinite center and the $4$-manifold is not homeomorphic to the 
mapping torus of some periodic diffeomorphism of an elliptic $3$-manifold, the homotopy class of the principal orbits of any locally linear, fixed-point free $\s^1$-action must have infinite order.

(2) A special case of Theorem 1.3, where $X$ is symplectic and the $\s^1$-action is free, was due
to Kotschick \cite{Ko1}. On the other hand, Theorem 1.3 is a strengthening of a result of 
Baldridge (cf. \cite{Bald3}, Theorem 1.1), which states that if a $4$-manifold with $b_2^{+}>1$ 
admits a smooth $\s^1$-action having at least one fixed point, then the Seiberg-Witten invariant 
vanishes for all $Spin^c$-structures.

\vspace{3mm}

In view of the examples in Theorem 1.1, it remains to see whether or not the answer to the following question would turn out to be affirmative (we formulate it here as a problem):

\vspace{3mm}

{\it Problem 3 {\em (}Question 4.4 in \cite{C}, Question 1.5(4) in \cite{C0}{\em)}:

\vspace{1.5mm}

Let $X$ be a simply connected smoothable $4$-manifold with even intersection form and non-zero signature. Prove that there exist a constant $C>0$ depending only on the homeomorphism type of $X$, 
such that for any prime number $p>C$, there are no $\Z_p$-actions on $X$ which are smooth with 
respect to some smooth structure on $X$.
}

\vspace{3mm}
 
The organization of the rest of the paper is as follows. In Section 2,
we consider $4$-manifolds with nonzero Seiberg-Witten invariant which admits a smooth, fixed-point 
free $\s^1$-action, and we show that the homotopy class of the principal orbits of the $\s^1$-action 
must have infinite order (here we are allowing more generally $b_2^{+}\geq 1$, cf. Theorem 2.1). 
Theorem 1.3 follows immediately from this result and the main result in Baldridge \cite{Bald3}.
The main technical ingredients involved in the proof in this section are Baldridge's work \cite{Bald2} 
on the Seiberg-Witten invariants of $4$-manifolds with a smooth fixed-point free $\s^1$-action, a 
de-singularization formula relating the Seiberg-Witten invariant of a $3$-orbifold with that of the 
underlying $3$-manifold \cite{C1}, and the recent work of Boileau, Leeb and Porti \cite{BLP} on the
geometrization of $3$-orbifolds. Section $3$ is devoted to a knot surgery
construction that kills all the smooth $\s^1$-actions on a given $4$-manifold. This naturally leads to
a proof for Theorem 1.2, and an equivariant version of it gives a proof for Theorem 1.1. 

\vspace{2mm}

This paper supersedes arXiv:1103.5681v3 [math.GT], ``Seifert fibered four-manifolds
with nonzero Seiberg-Witten invariant".

\section{New constraints of smooth $\s^1$-four-manifolds}

Recall that for an oriented $4$-manifold $M$ with $b^{+}_2>1$, the 
Seiberg-Witten invariant of $M$ is a map $SW_M$ that assigns to each $Spin^c$-structure $\L$ 
of $M$ an integer $SW_M(\L)$ which depends only on the diffeomorphism class of $M$ (cf. \cite{M}).  
In the case of $b_2^{+}=1$, the definition of $SW_M(\L)$ requires an additional choice of orientation 
of the $1$-dimensional space $H^{2,+}(M,\R)$, as discussed in Taubes \cite{T2}. A convention 
throughout this paper is that when we say $M$ has nonzero Seiberg-Witten invariant in the case of $b_2^{+}=1$, it is meant that the map $SW_M$ has nonzero image in $\Z$ for {\it some} choice of orientation of $H^{2,+}(M,\R)$. Note that under this convention, every symplectic $4$-manifold has
nonzero Seiberg-Witten invariant by the work of Taubes \cite{T1}.

The bulk of this section is devoted to proving the following theorem, which, with Theorem 1.1 in
Baldridge \cite{Bald3}, gives Theorem 1.3 immediately.

\begin{theorem}
Let $X$ be a $4$-manifold with $b_2^{+}\geq 1$ and nonzero Seiberg-Witten invariant. Then the homotopy class of the principal orbits of any smooth, fixed-point free $\s^1$-action on $X$ must have infinite order. 
\end{theorem}

Let $X$ be an oriented $4$-manifold equipped with a smooth, fixed-point free $\s^1$-action. The orbit 
map of the $\s^1$-action, $\pi:X\rightarrow Y$, defines a Seifert-type $\s^1$-fibration of the 
$4$-manifold, giving 
the orbit space $Y$ a structure of a closed, oriented $3$-dimensional orbifold whose singular set consists 
of a disjoint union of embedded circles, called {\it singular circles}. 
(Equivalently, the $4$-manifold is the total space of a principal $\s^1$-bundle over the $3$-orbifold.)

The first technical ingredient in the proof of Theorem 2.1 is the work of Baldridge \cite{Bald2}, which 
relates the Seiberg-Witten invariant of the $\s^1$-four-manifold with the base $3$-orbifold. More 
concretely, we will need the following lemma. 

\begin{lemma}
Suppose $X$ has nonzero Seiberg-Witten invariant. Let $\pi:X\rightarrow Y$ be the orbit map of a
smooth, fixed-point free $\s^1$-action on $X$. Then the $3$-orbifold $Y$ has nonzero Seiberg-Witten 
invariant provided that $b_1(Y)>1$. 
\end{lemma}

\begin{proof}
Let $\L$ be a $Spin^c$-structure such that $SW_X(\L)\neq 0$. Consider first the case where 
$b_2^{+}(X)>1$. In this case, by Theorem C in Baldridge \cite{Bald2}, $\L=\pi^\ast\L_0$ for some 
$Spin^c$-structure $\L_0$ on $Y$, and moreover
$$
SW_X(\L)=\sum_{\L^\prime\equiv \L_0\mod{\chi }}SW_Y(\L^\prime),
$$
where $\chi$ stands for the Euler class of $\pi$. It follows
readily that $Y$ has nonzero Seierg-Witten invariant in this case.

When $b_2^{+}(X)=1$ and $b_1(Y)>1$, the above formula continues to hold  (cf. \cite{Bald2}, 
Corollary 25) as long as $\L=\pi^\ast\L_0$. Thus it remains to show that $\L=\pi^\ast\L_0$ 
for some $Spin^c$-structure $\L_0$ on $Y$.

In order to see this, we observe first that $H^2(X;\Z)/Tor$ has rank $2$, and any torsion element of 
$H^2(X;\Z)$ is the $c_1$ of the pull-back of an orbifold complex line bundle on $Y$.
Moreover, there exists an embedded loop $\gamma$ lying in the complement of the singular set 
of $Y$, such that an element of $H^2(X;\Z)$ is the $c_1$ of the pull-back of an 
orbifold complex line bundle on $Y$ if and only if it is a multiple of the Poincar\'{e} dual of 
the $2$-torus $T\equiv \pi^{-1}(\gamma)\subset X$ in $H^2(X;\Z)/Tor$, see Baldridge \cite{Bald2}, 
Theorem 9.  With this understood, $\L=\pi^\ast \L_0$ for some $Spin^c$-structure $\L_0$ on $Y$ 
if and only if $c_1(\L)$ is Poincar\'{e} dual to a multiple of $T$ over $\Q$. 

We shall prove this using the product formula of Seiberg-Witten invariants in Taubes \cite{T2}. To 
this end, we let $N$ be the boundary of a regular neighborhood of $T$ in $X$. Then $N$ is a 
$3$-torus, essential in the sense of \cite{T2}, which splits $X$ into two pieces $X_{+}$, $X_{-}$. 
With this understood, Theorem 2.7 of \cite{T2} asserts that $SW_X(\L)$ can be expressed as a 
sum of products of the Seiberg-Witten invariants of $X_{+}$ and $X_{-}$. In particular, $c_1(\L)$ 
is expressed as a sum of $x$, $y$, where $x,y$ are in the images of $H^2(X_{+},N;\Z)$ and $H^2(X_{-},N;\Z)$ in $H^2(X;\Z)$ respectively. It is easily seen that mod torsion elements both are generated by the Poincar\'{e} dual of $T$. Hence the claim $\L=\pi^\ast \L_0$.

\end{proof}

The second technical ingredient is a formula in \cite{C1}, which relates the Seiberg-Witten invariant
of a $3$-orbifold with that of its underlying $3$-manifold. More precisely, let $Y$ be a closed,
oriented $3$-orbifold with $b_1(Y)>1$, whose singular set consists of a disjoint union of embedded
circles. Consider any singular circle $\gamma$, with multiplicity $\alpha>1$, and let $Y_0$ be the 
$3$-orbifold obtained from $Y$ by changing the multiplicity of $\gamma$ to $1$. The following theorem 
from \cite{C1} relates the Seiberg-Witten invariants of $Y$ and $Y_0$. 

\begin{theorem}
{\em (}De-singularization Formula, \cite{C1}{\em)} 
Let $S_Y,S_{Y_0}$ denote the set of $Spin^c$-structures on $Y,Y_0$ respectively. 
There exists a canonical map $\phi:S_Y\rightarrow S_{Y_0}$, which is surjective and $\alpha$ to $1$,
such that the Seiberg-Witten invariants of $Y$ and $Y_0$ obey the following equations 
$$
SW_Y(\xi)=SW_{Y_0}(\phi(\xi)), \;\; \forall \xi\in S_Y.
$$
\end{theorem}

Since the singular circles in $Y$ are of co-dimension two, the underlying space $|Y|$ is naturally a
$3$-manifold, which can be obtained from $Y$ by finitely many successive applications of the above procedure. In particular, above theorem implies that if the underlying $3$-manifold $|Y|$ has zero
Seiberg-Witten invariant, e.g., $|Y|$ contains a non-separating $2$-sphere, then so does the 
$3$-orbifold $Y$. Combined with Lemma 2.2, the following corollary follows easily.

\begin{corollary}
Let $\pi:X\rightarrow Y$ be a Seifert-type $\s^1$-fibration where $b_1(Y)>1$. If the underlying 
$3$-manifold $|Y|$ contains a non-separating $2$-sphere, then $X$ has vanishing 
Seiberg-Witten invariant. 
\end{corollary}

Recall that a $3$-orbifold is called {\it pseudo-good} if it contains no bad $2$-suborbifold. Our third
technical ingredient is the recent resolution of Thurston's Geometrization Conjecture for $3$-orbifolds
due to Boileau, Leeb and Porti \cite{BLP}. In particular, when combined with a result of McCullough 
and Miller \cite{MM}, their theorem implies that every pseudo-good $3$-orbifold admits a finite, regular manifold cover. The following lemma shows that the $3$-orbifolds arising in the present context are 
pseudo-good. 

\begin{lemma}
Let $X$ be a $4$-manifold with $b_2^{+}\geq 1$ which admits a smooth, fixed-point free $\s^1$-action,
and let $\pi:X\rightarrow Y$ be the corresponding Seifert-type $\s^1$-fibration. Then $Y$ is pseudo-good 
if either the Euler class of $\pi$ is torsion, or $X$ has nonzero Seiberg-Witten invariant.
\end{lemma}

\begin{proof}
We shall show that if $Y$ has a bad $2$-suborbifold $\Sigma$, then the Euler class of 
$\pi$ must be non-torsion and $X$ has vanishing Seiberg-Witten invariant. By definition, 
$\Sigma$ is an embedded $2$-sphere in the underlying $3$-manifold $|Y|$, such that either 
$\Sigma$ contains exactly one singular point of $Y$ or $\Sigma$ contains two singular points of different multiplicities. For simplicity, we let $p_1,p_2$ be the singular points on $\Sigma$, with 
$p_1, p_2$ contained in the singular circles $\gamma_1, \gamma_2$ of multiplicities 
$\alpha_1,\alpha_2$ respectively. We assume that $\alpha_1<\alpha_2$, 
with $\alpha_1=1$ representing the case where $\Sigma$ contains only one singular point 
$p_2$. Note that since $\alpha_1\neq \alpha_2$, $\gamma_1, \gamma_2$ are distinct.
 
We first show that the Euler class of $\pi$ must be non-torsion. Let $e(\pi)$ denote the image of
the Euler class of $\pi$ in $H^2(|Y|;\Q)$.
Then 
$$
\int_{|\Sigma|} e(\pi)=b+\beta_1/\alpha_1+\beta_2/\alpha_2, 
$$
where $b\in\Z, 1\leq \beta_1<\alpha_1$ or $\alpha_1=\beta_1=1$,  and 
$1\leq \beta_2<\alpha_2$ with $\text{gcd }(\alpha_2,\beta_2)=1$. 
It follows from the fact that $\alpha_1<\alpha_2$ and 
$\alpha_2, \beta_2$ are co-prime that $\int_{|\Sigma|} e(\pi)\neq 0$.
Hence the Euler class of $\pi$ is non-torsion. 

Now note that $b_1(Y)=b_2^{+}+1>1$ since the Euler class of $\pi$ is non-torsion (cf. \cite{Bald2}).
On the other hand, clearly $|\Sigma|$ is a non-separating $2$-sphere in $|Y|$. By Corollary 2.4,
$X$ has vanishing Seiberg-Witten invariant.

\end{proof}

\noindent{\bf Proof of Theorem 2.1}

\vspace{1mm}

We first introduce some notations. Let $\pi: X\rightarrow Y$ be the Seifert-type $\s^1$-fibration 
associated
to the smooth fixed-point free $\s^1$-action on $X$. Let $y_0\in Y$ be a regular point and $F$ be the 
fiber of $\pi$ over $y_0$, and fix a point $x_0\in F$. Denote by $i: F\hookrightarrow X$ the inclusion. 
We shall prove that $i_\ast: \pi_1(F,x_0)\rightarrow \pi_1(X,x_0)$ is injective. 

By Lemma 2.5, $Y$ is pseudo-good. As a corollary of the resolution of Thurston's Geometrization 
Conjecture for $3$-orbifolds (cf. \cite{BLP, MM}), $Y$ is very good, i.e, there is a $3$-manifold 
$\tilde{Y}$ with a finite group action $G$ such that $Y=\tilde{Y}/G$. Let $pr:\tilde{Y}\rightarrow Y$ 
be the quotient map, and let $\tilde{X}$ be the $4$-manifold which is the total space of the 
pull-back circle bundle of $\pi:X\rightarrow Y$ via $pr$.  Then $\tilde{X}$ has a natural free 
$G$-action such that $X=\tilde{X}/G$. We fix a point $\tilde{y}_0\in\tilde{Y}$ in the pre-image 
of $y_0$, and let $\tilde{F}\subset \tilde{X}$ be the fiber over $\tilde{y}_0$. We fix a 
$\tilde{x}_0 \in \tilde{F}$ which is sent to $x_0$ under $\tilde{X}\rightarrow X$.
Then consider the following commutative diagram 
$$
\begin{array}{llllllll}
\rightarrow & \pi_2(\tilde{Y},\tilde{y}_0) & \stackrel{\tilde{\delta}}{\rightarrow} & \pi_1(\tilde{F}, \tilde{x}_0) &
\stackrel{\tilde{i}_\ast}{\rightarrow} & \pi_1(\tilde{X},\tilde{x}_0) & \stackrel{\tilde{\pi}_\ast}{\rightarrow} & 
\pi_1(\tilde{Y},\tilde{y}_0) \rightarrow \\
                  &  \;\;\;\;\;\; \downarrow                    &                                              &   \;\;\;\;\;\;\parallel & 
                   & \;\;\;\;\; \;\downarrow &  & \;\;\;\;\;\;\downarrow \\
\rightarrow & \pi_2^{orb}(Y,y_0) & \stackrel{\delta}{\rightarrow} & \pi_1(F,x_0) &
\stackrel{i_\ast}{\rightarrow} & \pi_1(X,x_0) & \stackrel{\pi_\ast}{\rightarrow} & 
\pi_1^{orb}(Y,y_0) \rightarrow .\\                  
\end{array}
$$
Since $\pi_1(\tilde{X},\tilde{x}_0)\rightarrow \pi_1(X,x_0)$ is injective, it follows that $i_\ast$ is
injective if $\tilde{\delta}$ has zero image.

With the preceding understood, by the Equivariant Sphere Theorem of Meeks and Yau 
(cf. \cite{MY}, p. 480), there are disjoint, embedded $2$-spheres $\tilde{\Sigma}_i$ in $\tilde{Y}$ 
such that the union of $\tilde{\Sigma}_i$ is invariant under the $G$-action and the classes 
of $\tilde{\Sigma}_i$ generate $\pi_2(\tilde{Y})$ as a $\pi_1(\tilde{Y})$-module. Suppose the 
image of $\tilde{\delta}$ is non-zero. Then there is a $\tilde{\Sigma}\in\{\tilde{\Sigma}_i\}$ which
is not in the kernel of $\tilde{\delta}$. It follows that the Euler class of 
$\tilde{\pi}:\tilde{X}\rightarrow \tilde{Y}$ is nonzero on $\tilde{\Sigma}$, and consequently,
$\tilde{\Sigma}$ is a non-separating $2$-sphere in $\tilde{Y}$. It also follows that the Euler class
of $\pi$ is non-torsion, and in this case, we have $b_1(Y)=b_2^{+}+1>1$.

We first consider the case where there are no elements of $G$ which leaves the $2$-sphere 
$\tilde{\Sigma}$ invariant. In this case the image of $\tilde{\Sigma}$ in $Y$ under the quotient 
$\tilde{Y}\rightarrow Y$ is an embedded non-separating $2$-sphere which contains no singular 
points of $Y$.  By Corollary 2.4, $X$ has vanishing Seiberg-Witten invariant, a contradiction.

Now suppose that there are nontrivial elements of $G$ which leave $\tilde{\Sigma}$ invariant. 
We denote by $G_0$ the maximal subgroup of $G$ which leaves $\tilde{\Sigma}$ invariant. We shall 
first argue that the action of $G_0$ on $\tilde{\Sigma}$ is orientation-preserving. Suppose not, and 
let $\tau\in G_0$ be an involution which acts on $\tilde{\Sigma}$ reversing the orientation. Then 
since $\tau$ preserves the Euler class of $\tilde{\pi}$ and reverses the orientation of $\tilde{\Sigma}$, 
the Euler class of $\tilde{\pi}$ evaluates to $0$ on $\tilde{\Sigma}$, which is a contradiction. Hence 
the action of $G_0$ on $\tilde{\Sigma}$ is orientation-preserving.

Let $\Sigma$ be the image of $\tilde{\Sigma}$ in $Y$. Then $\Sigma$ is a non-separating, 
spherical $2$-suborbifold of $Y$. Consequently $|Y|$ contains a non-separating $2$-sphere,
which, by Corollary 2.4, implies that $X$ has vanishing Seiberg-Witten invariant, a contradiction.

This completes the proof of Theorem 2.1.

\vspace{3mm}

We end this section with a theorem concerning locally linear $\s^1$-actions on $X$ (when $b_2^{+}>1$). The result, when combined with Theorem 1.6 in \cite{C2}, implies the generalization of Theorem 1.3 we mentioned in the Introduction. 

\begin{theorem}
Let $X$ be a smooth $\s^1$-four-manifold with $b_2^{+}>1$ and nonzero Seiberg-Witten
invariant. Then every locally linear $\s^1$-action on $X$ is fixed-point free. 
\end{theorem}

\begin{proof}
The proof is based on a generalized version of Theorem 2.1 in Baldridge \cite{Bald3}, which states 
that if a $4$-manifold with $b_2^{+}>0$ admits a smooth $\s^1$-action with at least one fixed point, 
then the manifold contains an essential embedded $2$-sphere of non-negative self-intersection. 
Although Baldridge worked in the smooth category, his arguments remain valid in the locally linear
category. In particular, if the Hurwitz map $\pi_2\rightarrow H_2$ is finite, then every locally linear 
$\s^1$-action is fixed-point free. 

With the preceding understood, we claim that there exist embedded $2$-spheres $C_1,\cdots,C_N$ 
of self-intersection $0$ which generate the image of $\pi_2(X)\rightarrow H_2(X)$. Assume the 
claim momentarily. Then since $X$ has $b_2^{+}>1$ and nonzero Seiberg-Witten invariant, a theorem of Fintushel and Stern (cf. \cite{FS}) asserts that each $C_i$ must be torsion in $H_2(X)$, from which Theorem 2.6 follows. 

To see the existence of $C_1,\cdots,C_N$, we shall continue with the set-up and the notations 
introduced in the proof of Theorem 2.1. With this understood, note that the restriction of $\tilde{\pi}$ 
over each $\tilde{\Sigma}_i$ must be trivial from the proof of Theorem 2.1. We claim that there are 
sections $\Sigma_i$ of $\tilde{\pi}$ over $\tilde{\Sigma}_i$ or a push-off of it, such that no nontrivial 
element of $G$ will leave $\Sigma_i$ invariant. The image of such a $\Sigma_i$ under 
$\tilde{X}\rightarrow X$ is an embedded $2$-sphere of self-intersection $0$,  which will be 
our $C_i$. Since $\tilde{\pi}_\ast: \pi_2(\tilde{X})\rightarrow \pi_2(\tilde{Y})$ is isomorphic and 
$\pi_2(\tilde{X})=\pi_2(X)$, the $2$-spheres $C_i$ generate the image of 
$\pi_2(X)\rightarrow H_2(X)$.

To construct these $\Sigma_i$'s, note that there are three possibilities: (i) no nontrivial element
of $G$ leaves $\Sigma_i$ invariant, (ii) there is a nontrivial $g\in G$ (and no other elements) which
acts on $\Sigma_i$ preserving the orientation, and (iii) there is an involution $\tau\in G$ (and no
other elements) which acts on $\Sigma_i$ reversing the orientation. The existence of $\Sigma_i$
is trivial in case (i). In case (ii), any section of $\tilde{\pi}$ will do because $g$ acts as translations
on the fiber of $\tilde{\pi}$. In case (iii), we take $\Sigma_i$ to be a section of $\tilde{\pi}$ over
a push-off of $\tilde{\Sigma}_i$.  This shows the existence of $\Sigma_i$'s, and the proof of 
Theorem 2.6 is completed.

\end{proof}

\section{Killing smooth $\s^1$-actions and Hurwitz-type bound}

We begin by reviewing the knot surgery and the knot surgery formula for Seiberg-Witten invariants,
following Fintushel-Stern \cite{FinS}. To this end, let $M$ be a smooth $4$-manifold with 
$b_2^{+}>1$, which possesses an essential embedded torus $T$ of self-intersection $0$. Given 
any knot $K\subset\s^3$, the knot surgery of $M$ along $T$ with knot $K$ is a smooth 
$4$-manifold $M_K$ constructed as follows:
$$
M_K\equiv M\setminus Nd(T)\cup_\phi (\s^3\setminus Nd(K))\times \s^1.
$$ 
We remark that in the above definition of $M_K$ the diffeomorphism $\phi$ is required to send 
the meridian of $T$ to the longitude of $K$. As noted in \cite{FinS}, the choice of $\phi$ (up to
isotopy) is not unique, so that $M_K$ may not be uniquely determined in general. 
It follows easily from the Mayer-Vietoris sequence that the integral homology
of $M_K$ is naturally identified with the integral homology of $M$, under which the intersection 
pairings on $M$ and $M_K$ agree. (In \cite{FinS}, it is assumed that $M$ is simply connected, 
$T$ is $c$-embedded, and $\pi_1(M\setminus T)$ is trivial. These assumptions are irrelevant to 
the discussions here.) 

An important aspect of knot surgery is that the Seiberg-Witten invariant of $M_K$ can be computed
from that of $M$ and the Alexander polynomial of $K$, through the so-called knot surgery formula.
In order to state the formula, we let 
$$
\underline{SW}_M=\sum_z (\sum_{c_1(\L)=z} SW_M (\L)) t_z, \;\;\; z\in H^2(M;\Z),
$$
where $t_z\equiv \exp(z)$ is a formal variable, regarded as an element of the group ring 
$\Z H^2(M;\Z)$.

\begin{theorem}
{\em(}Knot Surgery Formula \cite{FinS}{\em)}
With the integral homology of $M$ and $M_K$ naturally identified, one has
$$
\underline{SW}_{M_K}=\underline{SW}_M \cdot \Delta_K(t),
$$
where $t=\exp(2[T])$. Here $\Delta_K(t)$ is the Alexander polynomial of $K$, and $[T]$ stands
for the Poincar\'{e} dual of the $2$-torus $T$. 
\end{theorem}

\noindent{\bf Remarks:}\hspace{1mm}
(1) The knot surgery formula was originally proved in \cite{FinS} under the assumption that 
$T$ is $c$-embedded. However, the assumption of $c$-embeddedness of $T$ is not essential
in the argument and it may be removed (cf. Fintushel \cite{F3}). 

(2) We note that by Theorem 3.1, $\underline{SW}_M\neq 0$ implies
$\underline{SW}_{M_K}\neq 0$ for any knot 
$K$, and $\underline{SW}_{M_K}=\underline{SW}_M$ if the Alexander polynomial of $K$ is trivial.
Moreover, if $H^2(M;\Z)$ has no $2$-torsions, so does $H^2(M_K;\Z)$, and in this case, $M$ and 
$M_K$ have the same Seiberg-Witten invariant provided that the Alexander polynomial of $K$ is trivial.

\vspace{2mm}

With the preceding understood, suppose we are given a fixed-point free, smooth 
$\s^1$-four-manifold $X$ with $b_2^{+}>1$. Let $\pi:X\rightarrow Y$ be the corresponding
Seifert-type $\s^1$-fibration. Furthermore, we assume that there is an embedded loop 
$l\subset X$ satisfying the following conditions:
\begin{itemize}
\item [{(i)}] $\pi(l)$ is an embedded loop lying in the complement of the singular set of $Y$, 
and $l$ is a section of $\pi$ over $\pi(l)$, 
\item [{(ii)}] no nontrivial powers of the homotopy class of $l$ are contained in the center of $\pi_1(X)$,
\item [{(iii)}] the $2$-torus $T\equiv \pi^{-1}(\pi(l))$ is non-torsion in $H_2(X)$. 
\end{itemize}

Now for any nontrivial knot $K\subset\s^3$, we consider a knot surgery manifold
$$
X_K\equiv X\setminus Nd(T)\cup_\phi (\s^3\setminus Nd(K))\times \s^1,
$$ 
where the diffeomorphism $\phi$ satisfies two further constraints: (a) $\phi$ sends a fiber of $\pi$ 
to the meridian of $K$, and (b) $\phi$ sends a push-off of $l$ (with respect to some framing) to
$\{pt\}\times \s^1$. 

With this understood, we have the following theorem.

\begin{theorem}
Suppose $\underline{SW}_X\neq 0$ and $K$ is not a torus knot. Then $X_K$ 
does not support any smooth $\s^1$-actions.
\end{theorem}

For a proof of Theorem 3.2, we introduce the following notations. Let $Y_1\equiv Y\setminus Nd(\pi(l))$ 
and $X_1\equiv X\setminus Nd(T)$, where $T$ is the $2$-torus $T=\pi^{-1}(\pi(l))$, such
that $X_1$ is the restriction of $\pi$ to the $3$-orbifold $Y_1$. 
Note that the boundary $\partial X_1$ is a $3$-torus $T^3$. We fix three embedded loops in 
$\partial X_1$: $m$, a meridian of $T$, $h$, a fiber of $\pi$, and $l^\prime$, a push-off of $l$ 
with respect to the framing used in the definition of $X_K$. Note that $m,h,l^\prime$ all together 
generate $\pi_1(\partial X_1)$. Furthermore, we assume $m$, $l^\prime$ are sections over 
$\pi(m)$, $\pi(l^\prime)$ respectively. Finally, recall that the underlying manifolds of $Y_1$, $Y$ are
denoted by $|Y_1|$, $|Y|$ respectively. 

\begin{lemma}
The map $i_\ast:\pi_1(\partial X_1)\rightarrow \pi_1(X_1)$ induced by the inclusion is injective.
\end{lemma}

\begin{proof}
Suppose to the contrary that $i_\ast:\pi_1(\partial X_1)\rightarrow \pi_1(X_1)$ has a nontrivial 
kernel. Then $\pi_1(\partial Y_1)\rightarrow \pi_1(|Y_1|)$ must also have a nontrivial kernel. To see this, suppose $\gamma\neq 0$ lies in the kernel of $i_\ast$. Then $\pi_\ast (\gamma)$ lies in the kernel of 
$\pi_1(\partial Y_1)\rightarrow \pi_1^{orb}(Y_1)$. On the other hand, if 
$\pi_\ast (\gamma)=0$, then $\gamma$ is a multiple of $[h]$, which is zero in $\pi_1(X_1)$
by assumption. This contradicts to the fact that $[h]$ has infinite order in $\pi_1(X)$ 
(cf. Theorem 1.3).  Hence $\pi_1(\partial Y_1)\rightarrow \pi_1^{orb}(Y_1)\rightarrow \pi_1(|Y_1|)$ 
has a nontrivial kernel. 

Now by the Loop Theorem, $\partial Y_1$ is compressible in $|Y_1|$. This means that there is an 
embedded disc $D\subset |Y_1|$ such that (1) $D\cap \partial Y_1=\partial D$, (2) $\partial D$ is a homotopically nontrivial simple closed loop in $\partial Y_1$. We claim that $\partial D$ must be a 
copy of the meridian $\pi(m)$ of $\pi(l)$ in $Y$. 

To see this, write $[\partial D]=s \cdot [\pi(m)] +t \cdot [\pi(l^\prime)]$ in $\pi_1(\partial Y_1)$.
If $t\neq 0$, then $[\pi(l)]$ has finite order in $\pi_1^{orb}(Y)$. This implies that a nontrivial 
power of $[l]$ lies in the subgroup generated by $[h]$, which lies in the center of $\pi_1(X)$, 
a contradiction to (ii). Hence $t=0$, and $s=1$ with $[\partial D]=[\pi(m)]$, so that $\partial D$ is 
isotopic to $\pi(m)$. 

With the preceding understood, the median $\pi(m)$ bounds an embedded disk 
$D\subset |Y_1|$. It follows that there is a non-separating $2$-sphere in $|Y|$ intersecting 
with $\pi(l)$ in exactly one point. By Corollary 2.4, $X$ has vanishing Seiberg-Witten invariant, 
a contradiction to $\underline{SW}_X\neq 0$. Hence Lemma 3.3.

\end{proof}

\noindent{\bf Proof of Theorem 3.2}

\vspace{2mm}

Since $K$ is a nontrivial knot, $\pi_1(\partial((\s^3\setminus Nd(K))\times \s^1))\rightarrow \pi_1((\s^3\setminus Nd(K))\times \s^1)$ is injective. With Lemma 3.3, this implies that 
$\pi_1(X_K)$ is an amalgamated free product 
$$
\pi_1(X_K)= \pi_1(X_1) \ast _{\pi_1(T^3)} \pi_1((\s^3\setminus Nd(K))\times \s^1),
$$
where $\pi_1(T^3)=\pi_1(\partial X_1)=\pi_1(\partial ((\s^3\setminus Nd(K))\times \s^1))$. 

Since $K$ is not a torus knot, the knot group of $K$ has trivial center (cf. \cite{Rolf}). This implies
that the center of $\pi_1((\s^3\setminus Nd(K))\times \s^1)$ is generated by the class of 
$\{pt\}\times \s^1$. On the other hand, by the construction of $X_K$, $\{pt\}\times \s^1$ is identified 
with $l^\prime$, whose class equals the class of $l$ modulo the class of $m$. By condition (ii),
no nontrivial powers of the homotopy class of $l$ is contained in the center of $\pi_1(X)$. 
Consequently, $\pi_1(X_K)$ has trivial center. Now $\underline{SW}_X\neq 0$ implies 
$\underline{SW}_{X_K}\neq 0$; in particular, $X_K$ has nonzero Seiberg-Witten invariant. 
By Theorem 1.3, $X_K$ does not support any smooth $\s^1$-actions. This completes the
proof of Theorem 3.2. 

\vspace{3mm}

The remaining of this section is occupied by the proofs of Theorems 1.1 and 1.2. 

\vspace{3mm}

\noindent{\bf Proof of Theorem 1.2}

\vspace{3mm}

Let $X$ be a smooth $\s^1$-four-manifold with $b_2^{+}>1$ and nonzero Seiberg-Witten
invariant, such that $H^2(X;\Z)$ has no $2$-torsions. We note that (i) the $\s^1$-action is 
fixed-point free, (ii) $\underline{SW}_X\neq 0$. Moreover, by Theorem 1.3, $\pi_1(X)$ has 
infinite center. For simplicity, we denote the center of a group $G$ by $z(G)$.

{\bf Case (1):} $\text{rank }z(\pi_1(X))=1$. In this case, we shall apply Theorem 3.2 to $X$ 
with a nontrivial knot $K$ whose Alexander polynomial is trivial. (Note that such a $K$ can not 
be a torus knot.) It suffices to show the existence of an embedded loop $l\subset X$ satisfying 
the conditions (i)-(iii). 

To this end, let $\pi: X\rightarrow Y$ be the corresponding Seifert-type $\s^1$-fibration. If 
the Euler class of $\pi$ is non-torsion, then $b_1(Y)=b_2^{+}(X)+1>2$, and if the Euler class
is torsion, one has $b_1(Y)=b_2^{+}(X)>1$, see Baldridge \cite{Bald2}, Theorem 9.  With this 
understood, it follows easily that there is an embedded loop $\gamma$ lying in the complement
of the singular set of $Y$, such that the class of $\gamma$ is not a multiple of the Poincar\'{e}
dual of the Euler class of $\pi$ over $\Q$. Let $T=\pi^{-1}(\gamma)\subset X$ and let $l$ be
any section of $\pi$ over $\gamma$. Then $T$ is non-torsion (cf. Baldridge \cite{Bald2}, Theorem 9)
so that (iii) is satisfied. Moreover, (i) is clearly true. Finally, (ii) follows from the fact that the 
homotopy class of $l$ is infinite order and the non-torsion part of the center $z(\pi_1(X))$ is
isomorphic to $\Z$ containing the regular fiber class of $\pi$ (cf. Theorem 1.3). 

{\bf Case (2):} $\text{rank }z(\pi_1(X))>1$. In this case, we shall invoke the following smooth classification proven in \cite{C2} (see Theorem 4.3 therein):
\begin{itemize}
\item [{(a)}] If $\text{rank }z(\pi_1(X))>2$, then $X$ is diffeomorphic to the $4$-torus $T^4$.
\item [{(b)}] If $\text{rank }z(\pi_1(X))=2$ and $\pi_2(X)\neq 0$, then $X$ is diffeomorphic to 
$T^2\times \s^2$.
\item [{(c)}] If $\text{rank }z(\pi_1(X))=2$ and $\pi_2(X)=0$, then $X$ is diffeomorphic to 
$\s^1\times N^3/G$, where $N^3$ is an irreducible Seifert $3$-manifold with infinite fundamental 
group, and $G$ is a finite cyclic group acting on $\s^1\times N^3$ preserving the product structure 
and orientation on each factor, and the Seifert fibration on $N^3$.
\end{itemize}

Note that case (b) is irrelevant here because of the assumption $b_2^{+}>1$. Consider case (c) 
where $X=\s^1\times N^3/G$. As it is given, $X$ admits a Seifert-type $\s^1$-fibration 
$\pi: X\rightarrow Y$ whose Euler class is torsion, where $Y=N^3/G$ admits a
Seifert fibration $pr: Y\rightarrow B$. With this understood, we note that $b_2^{+}(X)>1$
implies that $b_1(B)>0$. It follows easily that there is an embedded loop $\gamma\subset Y$ 
in the complement of the singular set of $Y$, such that $[\gamma]\neq 0$ in $H_1(|Y|;\Q)$
and $\gamma$ is not homotopic to a regular fiber of $pr: Y\rightarrow B$. A section of $\pi$
over $\gamma$, taken to be $l$, satisfies (i)-(iii), so that case (c) follows by Theorem 3.2. 

Finally, we consider case (a) where $X=T^4$. First, applying Theorem 3.2 with a nontrivial 
knot $K$ whose Alexander polynomial is trivial, we obtain a $4$-manifold $X^\prime$ 
diffeomorphic to $\s^1\times M$ where $M$ is a homology $T^3$ such that $z(\pi_1(M))=0$. 
In particular, $\text{rank }z(\pi_1(X^\prime))=1$. Applying Theorem 3.2 to $X^\prime$ as in
Case (1), we obtain a desired Fintushel-Stern knot surgery manifold of $X$. Hence Theorem 1.2. 

\vspace{2mm}

\noindent{\bf Proof of Theorem 1.1}

\vspace{2mm}

Consider the Kodaira-Thurston manifold $X=S^1\times N^3$, where $N^3= [0,1]\times T^2/\sim$ 
with $(0,x,y)\sim (1,x+y,y)$. The manifold $X$ is naturally a $T^2$-bundle over $T^2$, with $x,y$ 
being coordinates on the fiber. Note further that the translations along the $x$-direction define 
naturally a free smooth $\s^1$-action on $X$. The integral homology of $X$ is given as follows:
$$
H_1(X)=H_3(X)=\Z\oplus\Z\oplus\Z, \mbox{ and } H_2(X)=\Z\oplus\Z\oplus\Z\oplus\Z.
$$
In particular, $b_2^{+}=2$ and $H^2(X;\Z)$ has no $2$-torsions. Finally, as a $T^2$-bundle 
over $T^2$, $X$ has a symplectic structure with $c_1(K_X)=0$ (cf. Thurston \cite{Th}). By work of
Taubes \cite{T1}, $X$ has a unique Seiberg-Witten basic class $c_1(K_X)=0$.

In order to construct the $4$-manifolds $X_n$, $n\geq 2$, we shall apply Theorem 3.2 to $X$ in
an equivariant setting with some fixed nontrivial knot $K$ whose Alexander polynomial is trivial. 
More precisely, we consider the following embedded loop $l$ in $X$ which clearly satisfies the
conditions (i)-(iii) required in the construction: we pick any $T^2$-fiber in $X$ and take in that 
fiber an embedded loop parametrized by the $y$-coordinate. Notice that in the present case 
the $2$-torus $T\equiv \pi^{-1}(\pi(l))$ is simply the $T^2$-fiber we picked. 

Now note that the $\s^1$-factor in $X$ defines a natural $\s^1$-action, so that for any integer 
$n\geq 2$ there is an induced free smooth $\Z_n$-action on $X$, which preserves the 
$T^2$-bundle structure. With this understood, we let $X_n$ be a $4$-manifold resulted from 
a repeated application of the construction in Theorem 3.2 which is equivariant with respect 
to the free $\Z_n$-action. By the nature of definition, $X_n$ admits a smooth $\Z_n$-action.
To see that $X_n$ does not support any smooth $\s^1$-actions, we recall the fact that repeated 
knot surgery along parallel copies is equivalent to a single knot surgery using the connected sum 
of the knot (cf. Example 1.3 in \cite{C0}). Finally, it is clear that $X_n$ has the same integral
homology and intersection form and the same (and nonzero) Seiberg-Witten invariant of the 
Kodaira-Thurston manifold. Hence Theorem 1.1.

\vspace{3mm}

\noindent{\bf Acknowledgments}: I wish to thank Ron Fintushel for communications regarding the
knot surgery formula for Seiberg-Witten invariants, Dieter Kotschick for bringing to my attention 
his related work, and Inanc Baykur for helpful comments. Part of this work was done during
the author's visits to the Max Planck Institute for Mathematics, Bonn and the Institute for Advanced 
Study, Princeton. I wish to thank both institutes for their hospitality and financial support. Finally, 
I wish to thank Slawomir Kwasik for his timely advices, particularly his help with the related 
results in the literature.

\vspace{2mm}

{\Small University of Massachusetts, Amherst.\\
{\it E-mail:} wchen@math.umass.edu

\end{document}